\documentclass[12pt]{article}
\usepackage{amsmath, amsfonts, amsthm, amssymb, color, hyperref, extarrows, enumitem, nicefrac, blindtext, mathtools, comment, enumerate}

\newtheorem{theorem}{Theorem}[section]

\newtheorem{cor}[theorem]{Corollary}
\newtheorem{lem}[theorem]{Lemma}

\numberwithin{equation}{section}

\newtheorem{remark}[theorem]{Remark}

\newtheorem{example}{Example}

\usepackage[hyperpageref]{backref}

\usepackage{cite}

\hypersetup{
	colorlinks   = true,
	citecolor    = magenta}
	
\begin{document}
\title{\vspace{-1cm} \bf Unique continuation of Schr\"odinger-type equations for $\bar\partial$   II \rm}
\author{Yifei Pan  \ \ and \ \  Yuan Zhang}
\date{}

\maketitle

\begin{abstract}
In this paper, we extend our earlier unique continuation   results \cite{PZ2} for the Schr\"odinger-type inequality  $ |\bar\partial u| \le V|u|$ on a domain in $\mathbb C^n$ by removing    the smoothness assumption on solutions $u = (u_1, \ldots, u_N)$. More specifically, we    establish the unique continuation property  for  $W_{loc}^{1,1}$ solutions   when   the potential   $V\in L_{loc}^p $, $  p>2n$; and for $W_{loc}^{1,2n+\epsilon}$ solutions when $V\in L_{loc}^{2n}$ with   $N=1$ or $n = 2$.  Although    the unique continuation property  fails in general if  $V\in L_{loc}^{p}, p<2n$, we show that  the property still holds for $W_{loc}^{1,1}$ solutions  when   $V $ is a small constant multiple of $ \frac{1}{|z|}$.
  
\end{abstract}

\renewcommand{\thefootnote}{\fnsymbol{footnote}}
\footnotetext{\hspace*{-7mm}
\begin{tabular}{@{}r@{}p{16.5cm}@{}}
& 2020 Mathematics Subject Classification. Primary 32W05; Secondary 35J10.\\
& Key words and phrases. unique continuation, $\bar\partial$ operator, Schr\"odinger-type.
\end{tabular}}

\section{Introduction}
 Let $\Omega$ be a  domain in $\mathbb C^n$. Suppose $u: \Omega\rightarrow \mathbb C^N$ with $u\in W_{loc}^{1,1}(\Omega)$,  and satisfies the following Schr\"ordinger-type inequality \begin{equation}\label{eqn2}
     |\bar\partial u| \le V|u|\ \ \text{a.e. on}\ \ \Omega, 
\end{equation} for some nonnegative  locally Lebesgue integrable   function   $V$   on $\Omega$. This paper is a continuation of an earlier work \cite{PZ2} of the same authors on the unique continuation property for \eqref{eqn2}.  Specifically, we investigate  whether any Sobolev function $u$ satisfying \eqref{eqn2} vanishes identically if $u$ vanishes to infinite order in the $L^1$ sense  at some $z_0\in \Omega$. Here for   $q\ge 1$, $u\in L^q_{loc}(\Omega)$ is said to vanish to infinite order (or, be flat) in the $L^q$ sense at a point $z_0\in \Omega$, if for all $m\ge 1$,
\begin{equation*}\label{flat}
    \lim_{r\rightarrow 0} r^{-m}\int_{|z-z_0|<r}|u(z)|^q  dv_z =0.
\end{equation*}

While   Example \ref{pf}  indicates the general failure of the  unique continuation property  for $L_{loc}^p$ potentials with $p<2n$, it was shown in \cite{PZ2} the  property holds for smooth solutions to \eqref{eqn2} with $L_{loc}^p$ potentials, $p> 2n$, and  $L_{loc}^{2n}$ potentials if either $N=1$ or $n=2$.  Although the smooth category   remains   of primary interest,  our goal in this paper is to  extend the unique continuation property to  general Sobolev solutions.  

 We first show the unique continuation property holds for $ W_{loc}^{1,1}$ solutions to \eqref{eqn2}  when the potential $V\in L_{loc}^{p}$ for some $ p>2n$. In particular, this generalizes a unique continuation result of Bell and Lempert in \cite{BL90}  for bounded potentials.  See also a very recent result by Shi  \cite{Sh}     under some stronger assumptions on the potentials and the regularity of solutions. 
 
\begin{theorem}\label{main3}
    Let $\Omega$ be a  domain in $\mathbb C^n$. Suppose $u: \Omega\rightarrow \mathbb C^N$ with $u\in W_{loc}^{1,1}(\Omega)$,  and satisfies $ |\bar\partial u| \le V|u|$ a.e. on $\Omega$ for  $V \in L_{loc}^{p}(\Omega)$, $p>2n$.  If $u$ vanishes to infinite order  in the $L^1$ sense at some $z_0\in \Omega$, then $u$ vanishes identically.
\end{theorem}

  Similar to the approach in \cite{PZ2}, the proof utilizes a complex polar coordinate formula in Lemma \ref{cp} to convert the unique continuation problem in  higher dimensional domains to a known  one-dimensional result   in \cite{PZ} (see Theorem \ref{pz}). While the majority   of the work towards the proof was already  established in  \cite{PZ2}, unlike the  smooth  case,    the flatness of Sobolev solutions and the inequality \eqref{eqn2} do  not extend immediately to   their sliced counterparts.  The novelty in the proof is to show  that the minimal $W_{loc}^{1,1} $ Sobolev regularity  of $u$ is sufficient in ensuring   the sliced solutions to satisfy all the assumptions in Theorem \ref{pz}. In fact,     the assumptions $ u\in W_{loc}^{1,1}  $ and $V\in L^p_{loc} , p>2n$ here actually imply $ u\in W_{loc}^{1,p}  $, following  a standard boot-strap argument. This, combined with Lemma \ref{cgo}, transforms a $W_{loc}^{1, 1}$ solution  of \eqref{eqn2}     to   a family of $ W_{loc}^{1, 2} $ solutions to  some Schr\"odinger-type inequalities  along almost every complex radial direction. Moreover,  we  establish in Lemma \ref{fe}   an equivalence between the  $L^1$ flatness of $W_{loc}^{1, p}$ functions and a much stronger geometric flatness as in \eqref{hg}, which further passes the flatness  onto the sliced solutions.

When the potential $V\in L_{loc}^{2n}$,   we   obtain the following two unique continuation results: one for the case  $N=1$ (namely, $u$ is a scalar function), and the other for the case $n=2$. The integrability assumption of $V$ here   is sharp,  as indicated by Example \ref{pf}.   

\begin{theorem}\label{main5}
    Let $\Omega$ be a  domain in $\mathbb C^n$. Suppose $u: \Omega\rightarrow \mathbb C$ with $u\in W_{loc}^{1, 2n+\epsilon}(\Omega)$ for some $\epsilon>0$,  and satisfies $ |\bar\partial u| \le V|u|$ a.e. on $\Omega$ for some    $V \in L_{loc}^{2n}(\Omega)$.   If $u$ vanishes to infinite order in the $L^1$ sense  at some $z_0\in \Omega$, then $u$ vanishes identically.
\end{theorem}

\begin{theorem}\label{main6}
    Let $\Omega$ be a  domain in $\mathbb C^2$. Suppose $u: \Omega\rightarrow \mathbb C^N$ with $u\in W_{loc}^{1, 2n+\epsilon}(\Omega)$ for some $ \epsilon>0$,  and satisfies $ |\bar\partial u| \le V|u|$ a.e. on $\Omega$ for some    $V \in L_{loc}^{4}(\Omega)$.   If $u$ vanishes to infinite order in the $L^1$ sense  at some $z_0\in \Omega$, then $u$ vanishes identically.
\end{theorem}

 We have to assume $u\in W_{loc}^{1, 2n+\epsilon}$ in both theorems above  to apply Lemma \ref{cgo} and Lemma \ref{fe}. This is essentially due to the failure of the boot-strap argument   to  improve the Sobolev regularity of  solutions  when $V\in L_{loc}^{2n}$. It would be desirable to know whether  this regularity assumption on $u$ can be weakened further.   

On the other hand, it is worth noting that for the Laplacian  $\Delta$, the unique continuation property for $W^{2,2}_{loc}(\Omega)$ solutions of the   inequality   \begin{equation*}  
    |\Delta u|\leq V|\nabla u|\ \ \text{on}\ \ \Omega \subset \mathbb R^d
\end{equation*}  
with $V\in L_{loc}^{d}(\Omega)$  holds when $d=2, 3, 4$, but  fails in general when $d\ge 5.$   See the works of Chanillo-Sawyer \cite{CS90} and Wolff  \cite{Wo90, Wo94}. In contrast, Theorem \ref{main5} shows that, in the critical   $V\in  L_{loc}^{\dim_{\mathbb R}\Omega}(\Omega)$ case,  the unique continuation property for $\bar\partial$   is dimension-independent, highlighting a significant difference from the case of $\Delta$.

  Finally,   we explore   a special case where $V $ is  a constant multiple of $\frac{1}{|z|}$.  Note that $\frac{1}{|z|}\notin L^{2n}_{loc}$ near $0$. Although as in the $L^{2n}_{loc}$ potential  case  the boot-strap argument does not improve the Sobolev regularity of   solutions  near $0$ in general, thanks to Lemma \ref{hhl}, the additional  flatness of $u$ at $0$ allows us to eventually push $u$  to fall in $ W_{loc}^{1, q}$ for all $q<\infty$. In view of this, the $N=1$ case in the following theorem is  simply a direct consequence of    \cite[Theorem 5.1]{PZ2}  for $W_{loc}^{1,2}$ solutions. For $N\ge 2$ case, one can use a similar approach as in  Theorem \ref{main3} to weaken the smoothness assumption of $u$ in \cite[Theorem 1.5]{PZ2} to merely $W_{loc}^{1,1}$ regularity.  It should also be mentioned that  when   $N\ge 2$, the unique continuation   property fails in general if $C$ is  large, see Example \ref{ex2}.

\begin{theorem}\label{main4}
    Let $\Omega$ be a  domain in $\mathbb C^n$ and $0\in \Omega$. Let $u: \Omega\rightarrow \mathbb C^N$ with $u\in W_{loc}^{1,1}(\Omega)$,  and  satisfy $ |\bar\partial u| \le \frac{C}{|z|}|u|$ a.e. on $\Omega$. Assume $u$  vanishes to infinite order in the $L^1$ sense at $0\in \Omega$.\\
    1). If $N=1$, then $u $ vanishes identically.  \\
    2). If $N\ge 2$ and $C< \frac{1}{4 }$, then $u$ vanishes identically.
\end{theorem}

\section{Known results and examples}
In this section, we   list a few  known unique continuation properties for \eqref{eqn2}  that will be used in our paper, along with   counter-examples  for certain types of potentials. For clarification,  $u = (u_1, \ldots, u_N)\in  W^{1,1}_{loc}(\Omega)$ is said to satisfy the  inequality $ |\bar\partial u| \le V|u|$  a.e. on $\Omega$ if   \begin{equation*}  
     |\bar\partial u|: =\left(\sum_{j=1}^n\sum_{k=1}^N |\bar\partial_j u_k|^2\right)^{\frac{1}{2}}\le V \left(\sum_{k=1}^N|u_k|^2 \right)^{\frac{1}{2}}: = V|u| \ \  \text{a.e. on}\ \Omega. 
\end{equation*}
It is  worth pointing out that,  every $u\in W^{1,1}_{loc}(\Omega)$ satisfying $ |\bar\partial u| \le V|u|$ a.e. on $\Omega$ for some scalar function $V \in L^p_{loc}(\Omega)$  is  a weak solution to a Schr\"odinger-type equation of $\bar\partial$ below
\begin{equation*}\label{eqn3}
     \bar\partial u = u\mathcal V \ \ \text{on}\ \ \Omega
\end{equation*} 
for some $N\times N$ matrix-valued $(0, 1)$ form $\mathcal V = (\mathcal V_{jk})\in L^p_{loc}(\Omega)$, simply by letting $ \mathcal V_{jk}: = \frac{\overline{u_{j}}\bar\partial u_k}{|u|^2} $. 
\medskip

\begin{theorem}\cite{PZ}\label{pz}
Let $\Omega$ be a  domain in $\mathbb C^n$. Suppose $u: \Omega\rightarrow \mathbb C^N$ with $u\in W^{1, 2}_{loc}(\Omega)$,   and satisfies $ |\bar\partial u|\le V|u|$ a.e.\ on $\Omega$ for some $V\in L_{loc}^2(\Omega)$.\\ 
1). The weak unique continuation holds: if $u $ vanishes  in an open subset of $ \Omega$,   then $u$ vanishes identically.\\
 2). If $n=1$, then  the (strong) unique continuation holds: if $u$ vanishes to infinite order in the $L^2$ sense at some $z_0\in \Omega$, then $u$ vanishes identically.  
\end{theorem}

The unique continuation to \eqref{eqn2} fails in general if the potential does not belong to $L_{loc}^{2n}$. 

\begin{example}\label{pf}
For each $1\le p<2n$, let $\epsilon  \in (0, \frac{2n}{p}-1)$ and consider 
$$  |\bar\partial u| \le V|u|: = \frac{\epsilon }{2 |z|^{\epsilon+1}}|u|\ \ \text{on}\ \ B_1\subset \mathbb C^n.$$
 Note that $ V \in L^p(B_1)$, and $u_0= e^{-\frac{1}{|z|^\epsilon}}$ is a nontrivial smooth solution to the above equation that  vanishes to infinite order  at $0$. 

\end{example}

Despite Example \ref{pf}, the unique continuation property can still be expected for some special forms of potentials not in $L_{loc}^{2n}$, for instance, when the potential is a multiple of $\frac{1}{|z|}$.

\begin{theorem}\cite[Theorem 5.1]{PZ2}\label{33}
Let $\Omega$ be a   domain in $\mathbb C^n$ and $0\in \Omega$. Let $u: \Omega\rightarrow \mathbb C$ with $u\in  W^{1, 2}_{loc}(\Omega)$, and satisfies $|\bar\partial u|\le \frac{C}{|z|}|u|$ a.e. on $\Omega$ for some constant $C>0$. If $u$ vanishes to infinite order  in the $L^2$ sense at $0$, then $u $ vanishes identically. 
\end{theorem}

\begin{theorem}\cite[Theorem 6.1]{PZ2}\label{mm1}
Let $\Omega$ be a  domain in $\mathbb C$ and $0\in \Omega$. Let  $u: \Omega\rightarrow \mathbb C^N$ with $u\in W^{1, 2}_{loc}(\Omega)$, and  satisfy $|\bar\partial u|\le \frac{C}{|z|}|u| $ a.e. on $\Omega$ for some positive constant $C<\frac{1}{4 }$. If $u$ vanishes to infinite order in the $L^2$ sense  at  $0$, then $u $ vanishes identically.
\end{theorem}

In particular, when $N\ge 2$ and  $C$ is  large, the above unique continuation  property  no longer holds in general,  as indicated by an example   below  of the first author and  Wolff  \cite{PW98}, or \cite{AB94} by Alinhac and Baouendi. 

\begin{example}\label{ex2}
    Let   $v_0: \mathbb C\rightarrow \mathbb C$ be the  nontrivial smooth scalar function constructed in \cite{PW98} that vanishes to infinite order at $0$ and  satisfies $|\triangle v_0|\le \frac{C^\sharp}{|z|}|\nabla v_0|$  on $\mathbb C$ for some constant $C^\sharp>0 $. Letting  $u_0: = ( \partial \Re v_0, \partial \Im  v_0)$, then $u_0: \mathbb C\rightarrow \mathbb C^2$ is smooth,  vanishes to infinite order  at $0$, and  satisfies $|\bar\partial u_0|\le \frac{C^\sharp}{2|z|}|  u_0|$ on $\mathbb C$. 
\end{example}

\medskip

In the  case when the source dimension $n=1$, the unique continuation property  holds even when the potentials take on the following   hybrid forms involving  both   powers  of $\frac{1}{|z|}$ and    Lebesgue   functions. Note that none of these potentials  below  belongs to $L^2_{loc}$.

\begin{theorem}\cite[Theorem 5.5]{PZ2}\label{333}
Let $\Omega$ be a   domain in $\mathbb C$ containing $0$ and $1 < \beta < \infty$. Suppose $u: \Omega\rightarrow \mathbb C$ with $u\in  W^{1, 2}_{loc}(\Omega)$, and satisfies 
$$|\bar\partial u|\le   |z|^{-\frac{\beta-1}{\beta}}  V |u|\ \ \  \text{a.e.  on}\ \ \Omega   $$
 for some $   V\in L^{2\beta}_{loc}(\Omega)$. If $u$ vanishes to infinite order in the $L^2$ sense at $0$, then $u $ vanishes identically. 
\end{theorem}

\begin{theorem}\cite[Theorem 7.1]{PZ2}\label{main7}
      Let $\Omega$ be a  domain in $\mathbb C$ and $0\in \Omega$. Suppose   $u: \Omega\rightarrow \mathbb C^N$  with $u\in W^{1, 2}_{loc}(\Omega)$, and    satisfies \begin{equation*}
          |\bar\partial u| \le {|z|^{-\frac{1}{2}}} V |u| \ \ \text{a.e. on}\ \ \Omega 
      \end{equation*}    for some    $V \in L_{loc}^{4}(\Omega)$.  If $u$ vanishes to infinite order in the $L^2$ sense at $0$, then $u$ vanishes identically.
\end{theorem}

\section{Properties of sliced functions}\label{s4}
  As will be seen in the next section, we shall slice   Sobolev solutions    to  \eqref{eqn2} and the potential along almost all  complex one-dimensional radial directions. The key idea to justify this approach lies in   the following complex polar coordinate  formula,    whose proof can be found, for instance, in \cite[Lemma 4.2]{PZ2}. Denote by $S^{2n-1}$ the unit sphere in $\mathbb C^n$. Let $B_r$ be the open ball centered at $0$ of radius $r$ in $\mathbb C^n$, and $D_r$ be the open disk centered at $0$ of radius $r$ in $\mathbb C$.    

\begin{lem}\label{cp}
    Let  $u\in L^1(B_{r })$. Then  for  a.e. $\zeta\in S^{2n-1}$,   $|w|^{2n-2}u(w\zeta)$  as a function of $w\in D_{r }$ is in $   L^1(D_{r })$, with 
    $$\int_{|z|< r } u(z) dv_z = \frac{1}{2\pi} \int_{|\zeta| =1}  \int_{|w|< r }|w|^{2n-2}u(w\zeta)dv_wdS_\zeta.$$
\end{lem}
\medskip

\begin{cor}\label{coo}
  Let   $u\in L^p(B_r)$ for some $ p>n\   (=\dim_{\mathbb C}B_r)$. Then given a.e. $\zeta \in S^{2n-1}$, $v(w): =u(w\zeta)$ as a function of $w\in D_r$ is in $L^q(D_r)$ for all $1\le q<  \frac{p}{n}$. In particular,  if $u\in L^p(B_r)$ for some $ p >2n$, then for a.e. $\zeta\in S^{2n-1}$,  $v\in L^2(D_r)$.
\end{cor}

\begin{proof}
  by Lemma  \ref{cp}
$$\int_{|z|< r} |  u(z)|^p dv_z = \frac{1}{2\pi} \int_{|\zeta| =1}  \int_{|w|< r}|w|^{2n-2}|u(w\zeta)|^pdv_wdS_\zeta. $$
Then  for a.e. $\zeta\in S^{2n-1}$, \begin{equation}\label{hk}
    \int_{|w|< r}|w|^{2n-2}|u(w\zeta)|^pdv_w<\infty. 
\end{equation}  
Making use of  H\"older's inequality
\begin{equation*}\begin{split}
        \int_{|w|< r } \left|v(w)\right|^qdv_w &=  \int_{|w|< r } |w|^{\frac{(2n-2)q}{p}}\left|v(w)\right|^q\cdot  |w|^{-\frac{(2n-2)q}{p}}dv_w \\
        &\le  \left( \int_{|w|< r }  |w|^{2n-2}|u(w\zeta)|^p dv_w\right)^{\frac{q}{p}}  \left(\int_{|w|< r } |w|^{-\frac{(2n-2)q}{p}\frac{p}{p-q}}dv_w\right)^{\frac{p-q}{p}}\\
        &=\left( \int_{|w|< r}|w|^{2n-2}|u(w\zeta)|^pdv_w\right)^{\frac{q}{p}}  \left(\int_{|w|< r } |w|^{-\frac{(2n-2)q}{p-q}}dv_w\right)^{\frac{p-q}{p}}.
\end{split}
\end{equation*}
Since $q<  \frac{p}{n}$, we have $\frac{(2n-2)q}{p-q}<2 $ and thus $ \int_{|w|< r } |w|^{-\frac{(2n-2)q}{p-q}}dv_w<\infty$. This, combined with  \eqref{hk}, proves the corollary.
\end{proof}
\medskip

In order to  convert the unique continuation property   in the higher source dimensional case to   the complex one-dimensional case  where Theorem \ref{pz} can be applied, we first establish Lemma  \ref{cgo} below. This  allows us to obtain sufficient regularity for  the sliced functions when $u\in W_{loc}^{1, p}, p>2n $. As   demonstrated in Example \ref{exhh} in Section \ref{3}, the lemma does not hold for general $W^{1, 2n} $ functions.

\medskip

\begin{lem}\label{cgo}
     Suppose  $u\in W^{1, p}(B_{r })$ for some $p>2n\  (= \dim_{\mathbb R} B_{r })$. Then for a.e.  $\zeta\in S^{2n-1}$,   $v(w): = u(w\zeta)$  as a function of $w\in D_{r }$ belongs to $ W^{1, 2}(D_{r }).  $ Moreover,  \begin{equation}\label{cfo}
   \nabla v(w)  =   \zeta\cdot\nabla  u(w\zeta),\ \  \  w\in  D_{r }
\end{equation}  in the sense of distributions.
     
\end{lem}

\begin{proof}
   We only need to  show \eqref{cfo}. In fact, since $u\in W^{1, p}(B_r)$,  $p>2n$, if  \eqref{cfo} holds,  we can apply Corollary \ref{coo} to $u$ and $\nabla u$ respectively, and obtain $v\in W^{1,2}(D_r)$.

   Let $u_j\in C^\infty(B_r)\cap W^{1, p}(B_r)$ be such that $u_j\rightarrow u$ in the $ W^{1, p}(B_r)$ norm and $v_j(w): = u_j(w\zeta), w\in D_r$.   Then  \begin{equation}\label{hhh}
   \nabla v_j(w) = \zeta\cdot\nabla  u_j(w\zeta),  \ \ \ w\in D_r. 
\end{equation}  Since $u$ is continuous on $B_r$ (and so is $v$ on $D_r$),  by Sobolev embedding theorem   there exists some constant $C>0$ such that 
$$\|v_j- v\|_{C(D_{r})}\le \| u_j -u \|_{C(B_{r})}\le C\|u_j-u\|_{W^{1, p}(B_{r}) }\rightarrow 0 $$
as $j\rightarrow 0$. In particular, \begin{equation}\label{cho}
    v_j\rightarrow v \ \ \text{on}\ \ D_r
\end{equation}  in the sense of distributions. 

On the other hand, applying Lemma \ref{cp} to $|\nabla u-\nabla u_j|^p$,   the function $$  g_j(\zeta): =  \int_{|w|< r}|w|^{2n-2}|\nabla u(w\zeta)-\nabla u_j(w\zeta)|^pdv_w, \ \ \zeta\in S^{2n-1}$$ satisfies
\begin{equation*}
    \begin{split}
         \int_{|\zeta| =1} |g_j(\zeta)|dS_\zeta &=  \int_{|\zeta| =1}  \int_{|w|< r}|w|^{2n-2}|\nabla u(w\zeta)-\nabla u_j(w\zeta)|^pdv_wdS_\zeta\\
         &= 2\pi\int_{|z|< r}| \nabla u(z)-\nabla u_j(z)|^p dv_z\rightarrow 0 
    \end{split}
\end{equation*} 
   as $j\rightarrow \infty$.  Hence, by passing to a subsequence if necessary (see, for instance, \cite[Theorem 4.9]{Br}), one has for a.e. $\zeta\in S^{2n-1}$, 
   \begin{equation*} 
       g_j(\zeta)\rightarrow 0 
   \end{equation*}
  as $j\rightarrow 0$.   Making use of  H\"older's inequality
   \begin{equation*}
       \begin{split}
            \int_{|w|< r} | \zeta\cdot\nabla  u_j(w\zeta) - \zeta\cdot\nabla  u(w\zeta) |dv_w   \le  & \int_{|w|< r} |w|^{- \frac{2n-2}{p}}\cdot |w|^{\frac{2n-2}{p}} |  \nabla  u_j(w\zeta) - \nabla  u(w\zeta) |dv_w\\
          \le & \left(\int_{|w|< r} |w|^{- \frac{2n-2}{p-1}}\right)^\frac{p-1}{p}\cdot   (g_j(\zeta))^{\frac{1}{p}} \rightarrow 0
       \end{split}
   \end{equation*} 
   as $j\rightarrow \infty$. Here we used the fact that $p>2n$, so  $ \frac{2n-2}{p-1}<2$ particularly.
  This implies that \begin{equation}\label{cko}
      \zeta\cdot\nabla  u_j(w\zeta)\rightarrow \zeta\cdot\nabla  u(w\zeta)\ \ \text{on}\ \    D_r 
  \end{equation}   in the sense of distributions. \eqref{cfo} is thus proved in view of \eqref{hhh}, \eqref{cho} and \eqref{cko}.    
\end{proof}

\medskip

The following lemma allows us to   extend the   flatness   of $  W^{1, p}_{loc}, p>2n$ solutions  to   their  restrictions along the radial directions. 

\begin{lem}\label{fe}
    Suppose  $u\in W^{1, p}(B_{r})$ for some $p>2n (= \dim_{\mathbb R} B_{r})$. Then $u$ vanishes to infinite order in the $L^1$ sense at $0$ if and only if for every $m\ge 1$, \begin{equation}\label{hg}
    |u(z)|= O(|z|^m) \ \ \text{for all}\ \  |z|<<1.\end{equation} In particular, if $u$ vanishes to infinite order in the $L^1$ sense at $0\in B_{r}$, then for a.e.  $\zeta\in S^{2n-1}$,   $v(w): = u(w\zeta)$  as a function of $w\in D_{r}$ vanishes to infinite order in the $L^2$ sense at $0\in D_{r}$.
\end{lem}

\begin{proof}
    The backward direction is obvious by definition. In fact, one can also easily check that, for any given $q\ge 1$,   as long as  $u\in L^q$ near $0$ and satisfies \eqref{hg},   $u$  vanishes to infinite order in the $L^q$ sense at $0$.
    
    Assume $u$ vanishes to infinite order in the $L^1$ sense at $0$. Since   $p>2n$, by Sobolev embedding theorem  $u$ is continuous on $B_{r}$ with
      \begin{equation}\label{gg1}
        \sup_{|z|< {r} } |u | + \int_{|z|<r }|\nabla u|^{p}  \le C_0.   
    \end{equation}
   for some constant $C_0>0$. Fix some $q$ with $2n<q<p$. For every $m\ge 1$, by the continuity of $u$ and the $L^1$ flatness of $u$ at $0$, one has $u(0)=0$  and 
    \begin{equation}\label{hh}
        \int_{|z|<t} |u |^{\frac{ qp}{p-q}}\le  C_0^{\frac{ qp}{p-q}-1}\int_{|z|<t} |u | \le O(t^{\frac{2pqm}{ p-q}}), \ \ \text{for all}\ \ t<<1.
    \end{equation}
    Letting $v: = u^2$, then $v\in W^{1, p}(B_{r}) $ with $\nabla v = 2u \nabla u$ on $ B_{r}$. Moreover, making use of H\"older's inequality, \eqref{gg1} and \eqref{hh}, 
    \begin{equation}\label{gg2}
    \begin{split}
                \int_{|z|<t}|\nabla v|^q =&   2^q\int_{|z|<t}|u|^{ q}|\nabla u|^q \le 2^q\left(\int_{|z|<t} |u|^{\frac{ qp}{p-q}}\right)^{\frac{p-q }{ p}} \left(\int_{|z|<t}|\nabla u|^{p}\right)^\frac{q}{p}\\
        \le& 2^pC_0^\frac{q}{p}\left(\int_{|z|<t} |u|^{\frac{ qp}{p-q}}\right)^{\frac{p-q }{ p}}\le  O(t^{2qm}),      \ \ \text{for all}\ \ t<<1.    
    \end{split}
    \end{equation}

    On the other hand, since $v(0)=0$ and $q>2n$, there exists a constant $C_1>0$ such that 
    \begin{equation*}
      \sup_{|z|\le \frac{t}{2}} |v|\le C_1t^{1- \frac{2n}{q}}\left( \int_{B_{t}}|\nabla v|^q\right)^{\frac{1}{q}} ,\ \ \text{for all}\ \ t<  r.
    \end{equation*}
     See, for instance, \cite[pp. 283]{Ev}. Together with \eqref{gg2}, we get  
    \begin{equation*}
      |u(z)|^2= |v( z)|\le    C_1 \left( \int_{B_{2|z|}}|\nabla v|^q\right)^{\frac{1}{q}} \le   O(|z|^{2m}), \ \ \text{for all}\ \ |z|<<1.
    \end{equation*}
  This proves the forward direction.

  Finally, if $u$ vanishes to infinite order in the $L^1$ sense at $0$, then $u$ satisfies \eqref{hg} by the equivalence of the two types of  flatness. Hence for a.e. $\zeta\in S^{2n-1}$, \eqref{hg} holds true for $v$ as well. In particular, since $v\in L^2$ near $0$ due to Lemma \ref{cgo}, $v$ vanishes to infinite order in the $L^2$ sense at $0$. 
\end{proof}

We would like to note that, although both Lemma \ref{cgo} and Lemma \ref{fe} are stated for scalar functions for simplicity of notations, they can be seamlessly extended to the case of vector-valued functions.

\section{Proof of the main theorems}\label{3}

In this section we shall prove Theorems \ref{main3}-\ref{main4}. Let us start by stating  a local ellipticity lemma of $\bar\partial$, which will be repeatedly used in  the boot-strap argument.

\begin{lem}\cite[Lemma 3.1]{PZ2}\label{el}
    Let $\Omega$ be a   domain in $\mathbb C^n$  and $1<p<\infty$. Let $V \in L_{loc}^{p}(\Omega)$ be a $\bar\partial$-closed $(0,1)$ form on $\Omega$. Then every  solution to $\bar\partial f =V $ on $\Omega$ in the sense of distributions belongs to $ W_{loc}^{1, p}(\Omega) $. 
\end{lem}


\begin{proof}[Proof of Theorem \ref{main3}: ] Without loss of generality, let $z_0=0$, and $r$ be small such that $B_r\subset \Omega$. 
Since $u\in L_{loc}^{\frac{2n}{2n-1}}(B_r)$ by Sobolev embedding theorem, as an application of H\"older's inequality, $Vu\in L_{loc}^{\frac{2n }{(2n-1) +\frac{2n}{p}}}(B_r)$. It follows from Lemma \ref{el} and the inequality \eqref{eqn2}  that $u\in W_{loc}^{1, \frac{2n }{(2n-1) +\frac{2n}{p}}}(B_r)\subset L_{loc}^{\frac{2n}{(2n-1)+\frac{2n}{p}-1}}(B_r)$. Since $p>2n$, a boot-strap  argument as above can eventually give  $u\in W_{loc}^{1, p}(B_r)$.  

For each fixed $\zeta\in S^{2n-1}$, let  $v(w): = u(w\zeta)$ and $ \tilde V(w): = V(w\zeta), w\in D_{r}$. Since     $V\in L^p_{loc}(B_r), p>2n$, by   Corollary \ref{coo} we have   $\tilde V\in L^2_{loc}(D_r)$. On the other hand, it follows from Lemma \ref{fe} and the $L^1$ flatness of $u$ at $0$ that    $v$ vanishes to infinite order in the $L^2$ sense  at $0$. Moreover,     $v\in W_{loc}^{1, 2}(D_r)$ by Lemma \ref{cgo}, and as a consequence of \eqref{cfo}, 
$$|\bar\partial v(w) |= | \zeta\cdot \bar\partial  u(w\zeta)|\le V(w\zeta) |u(w\zeta)| = \tilde V(w) |v(w)|,\ \ w\in D_{r}. $$ 
Hence  we  can make use of  Theorem \ref{pz} part 2) to $v$ and obtain $v = 0$ on $D_{r}$ for a.e. $\zeta\in S^{2n-1}$. Thus $u=0$ on $B_{r}$. Apply the weak unique continuation property in Theorem \ref{pz} part 1) to further get $u\equiv 0$ on $\Omega$. 
\end{proof}

 \medskip

Before proving Theorems \ref{main5}-\ref{main6} for $L_{loc}^{2n}$ potentials, we point  out that the slicing method in the proof of Theorem \ref{main3} fails to work  in general if $V$ is merely in $ L^{2n}_{loc}$.   More precisely, there exists a $L_{loc}^{2n}$  potential $V$   whose  complex radial restriction  is not in $L_{loc}^2$.

\begin{example}\label{exhh}
Assume $n\ge 2$. For each $\epsilon\in \left(\frac{1}{2},  \frac{2n-1}{2n}\right)$, consider   $$u(z) = e^{-(-\ln |z|)^\epsilon}, \ \ z\in B_{\frac{1}{2}}\subset \mathbb C^n, $$  and $$V(z) = \frac{\epsilon (-\ln |z|)^{\epsilon-1}}{2|z|}, \ \ z\in B_{\frac{1}{2}}.$$ Then $u\in W^{1, 2n}(B_{\frac{1}{2}})$, $V\in L^{2n}(B_{\frac{1}{2}}  )$ and 
$$  |\bar\partial u| \le V|u|  \ \ \text{on}\ \ B_{\frac{1}{2}}.$$
(One can verify that $u$ only vanishes to a finite order at $0$ in the $L^q$ sense for every $q\ge 1$.) On the other hand,  for each $\zeta\in S^{2n-1}$, the complex radial restriction  of $V$ is $$ \tilde V(w): = V(w\zeta) = \frac{\epsilon (-\ln |w|)^{\epsilon-1}}{2|w|}, \ \ w\in D_{\frac{1}{2}}.$$ 
It is easy to see that $\tilde V\notin L^2_{loc}(D_\frac{1}{2})$ since $\epsilon> \frac{1}{2}$. 
\end{example}

\medskip

\begin{proof}[Proof of Theorems \ref{main5} and \ref{main6}: ]  Let $z_0=0$ and $r>0$ be small such that  $B_r\subset \Omega $.  For each fixed $\zeta\in S^{2n-1}$, let   $\tilde V(w): = |w|^{\frac{n-1}{n}}V(w\zeta)$   and $v(w): = u(w\zeta), w\in D_{r}$. It follows from  Lemma \ref{cp} that $\tilde V\in L_{loc}^{2n}(D_{r})$ for a.e. $\zeta\in S^{2n-1}$.  On the other hand, since $u\in W_{loc}^{1, p}(B_r)$ for some $p>2n$,  $v$ vanishes to infinite order at $0$ in the $L^2$ sense by Lemma \ref{fe}. Moreover, as a consequence of Lemma \ref{cgo},  $v\in W_{loc}^{1, 2}(D_r)$ and  satisfies
$$|\bar\partial v(w) | \le   |w|^{-\frac{n-1}{n}}\tilde V(w)|v(w)|,\ \ w\in D_{r}. $$
For a.e. $\zeta\in S^{2n-1}$,   employ   Theorem \ref{333} if $N=1$; employ Theorem \ref{main7} if $n=2$. In both cases, we get  $v = 0$ on $D_{r}$ and thus   $u=0$ on $B_{r}$. Apply the weak unique continuation property to get $u\equiv 0$.
\end{proof}
\medskip

\begin{remark}\label{re}
  In view of Theorems \ref{main3}-\ref{main6},  the following two questions still remain  open. With an approach similar as in the proof to Theorems \ref{main5}-\ref{main6}, the resolution of Question 1  is reduced to that of Question 2.\\
   \noindent\textbf{1.  } 
     Let $\Omega$ be a  domain in $\mathbb C^n, n\ge 3$ and $N\ge 2$. Suppose   $u: \Omega\rightarrow \mathbb C^N$ with $u\in W_{loc}^{1, p}(\Omega)$ for some $p>2n$ and   satisfies $ |\bar\partial u| \le V|u|$ a.e. on $\Omega$ for some    $V \in L_{loc}^{2n}(\Omega)$. If $u$ vanishes to infinite order in the $L^1$ sense at some $z_0\in \Omega$, does $u$ vanish  identically?  
        \medskip
        
     \noindent\textbf{2. }
     Let $\Omega$ be a  domain in $\mathbb C$ containing $0$, and  $n, N\in \mathbb Z^+$  with $n\ge 3, N \ge 2$. Suppose   $u: \Omega\rightarrow \mathbb C^N$ with $u\in W_{loc}^{1,2}(\Omega)$ and    satisfies $ |\bar\partial u| \le {|z|^{-\frac{n-1}{n}}} V |u|$ a.e. on $\Omega$ for some    $V \in L_{loc}^{2n}(\Omega)$.  If $u$ vanishes to infinite order in the $L^2$ sense at $0\in \Omega$, does  $u$ vanish  identically?

\end{remark}
\medskip

Next we prove Theorem \ref{main4} for $W^{1,1}_{loc}$ solutions to    $
    |\bar\partial u|\le \frac{C}{|z|}|u|$ a.e. on $\Omega$, where $C$ is a positive constant.  It is not hard to see that $u\in W^{1, q}$ everywhere away from $0$ for all $q < \infty$ using a boot-strap argument since $\frac{1}{|z|}\in L^\infty$ off $0$. However, the argument is not directly effective near  $0$ because $\frac{1}{|z|}\notin L^{2n}$ there.  The following two  lemmas show that     $u\in W^{1, q}$  for all $q < \infty$ near $0$ under the $L^1$ flatness assumption at $0$.  

 \begin{lem}\label{22}
Let $u\in L^1$ near $0\in \mathbb C^n$, and vanishes to infinite order in the $L^1$ sense at $0$. Then for each $M>0$, $\frac{u}{|z|^M} \in L^1$ near $0$, and vanishes to infinite order in the $L^1$ sense at $0$. 
\end{lem}

\begin{proof}
For each $m\ge 1$ and $\epsilon>0$, by  the $L^1$ flatness of $u$ at $0$,   
\begin{equation*}
    \int_{|z|<r}|u|dv_{z}\le \epsilon r^{m+M}  \ \ \text{for all}\ \ r<<1.
\end{equation*}
Then  
\begin{equation*}
    \begin{split}
      \int_{|z|<r}\frac{|u|}{|z|^{M}}dv_{z} =&     \sum_{j=1}^\infty \int_{\frac{r}{2^j}<|z|<\frac{r}{2^{j-1}}}\frac{|u|}{|z|^{M}}dv_{z} 
      \le   \sum_{j=1}^\infty \frac{2^{Mj}}{r^{M}}\int_{\frac{r}{2^j}<|z|<\frac{r}{2^{j-1}}}|u|dv_{z}
      \le  \sum_{j=1}^\infty \frac{2^{Mj}}{r^{M}}\int_{|z|<\frac{r}{2^{j-1}}}|u|dv_{z} \\
      \le&  \epsilon \sum_{j=1}^\infty \frac{2^{Mj}}{r^{M}}\frac{r^{m+M}}{2^{(m+M)(j-1)}} 
      = \epsilon2^{M}r^m\sum_{j=1}^\infty 2^{-m(j-1)}\le \epsilon2^{M+1} r^m  \ \ \text{for all}\ \ r<<1.
    \end{split}
\end{equation*}
In particular, $\frac{u}{|z|^M}\in L^1$ near $0$, and vanishes to infinite order in the $L^1$ sense at $0$.
\end{proof}


\begin{lem}\label{hhl}
  Let $\Omega$ be a  domain in $\mathbb C^n$ and $0\in \Omega$. Let $u: \Omega\rightarrow \mathbb C^N$ with $u\in W_{loc}^{1,1}(\Omega)$,  and  satisfy $ |\bar\partial u| \le \frac{C}{|z|}|u|$ a.e. on $\Omega$ for some constant $C>0$. Assume $u$  vanishes to infinite order  in the $L^1 $ sense  at $0\in \Omega$. Then $u\in W_{loc}^{1, q}(\Omega)$ for every  $q<\infty$.   
\end{lem}

\begin{proof}
    We first claim   if  $u\in L^{p}$ near $0$ for some $p>1$, then   $ \frac{u}{|z| }\in L^q$ for every $1<q<p$. Indeed, for each $q<p$, letting $\epsilon = \frac{p-q}{p-1}$, then $0<\epsilon<1$ and $ \frac{q-\epsilon}{ 1-\epsilon} =p $. By H\"older's inequality and Lemma \ref{22},
 $$\int_{|z|<r}\left| \frac{u }{|z| }\right|^q  = \int_{|z|<r} \frac{|u|^\epsilon}{|z|^{q }}\cdot |u|^{q-\epsilon} \le \left(\int_{|z|<r}  \frac{|u|}{|z|^{\frac{q }{\epsilon}}}\right)^{\epsilon} \left(\int_{|z|<r} |u|^{p}\right)^{1-\epsilon}<\infty, \ \ \ r<<1.$$  The claim is proved.

  We are now ready to employ the boot-strap argument as in the proof to Theorem \ref{main3}. Since $u\in W^{1,1}$ near $0$,   it follows from the Sobolev embedding theorem that  $u\in L^{\frac{2n}{2n-1}} $ near $0$. Consequently, the above  proved claim gives $\frac{u}{|z|}\in L^q$ near $0$ for any $q<\frac{2n}{2n-1}$.  Lemma \ref{el} further allows us to  obtain $u\in W^{1, q}\subset L^{q'}$  near $0$ for any $q'< \frac{2n}{2n-2}$. Repeating the process  eventually leads to $u\in W^{1,q}$ near $0$ for every $q<\infty$.  The fact that $u\in W^{1, q}$ near every other point than $0$ is proved in a similar but simpler  manner (without using the claim) since $\frac{1}{|z|}\in L^\infty$ near those points. 
\end{proof}
\medskip

\begin{proof}[Proof of Theorem \ref{main4}:]The $N =1$ case is an immediate consequence of Theorem \ref{33} and Lemma \ref{hhl}. So we assume  $N\ge 2$.  Let $r$ be small such that $B_{r}\subset   \Omega$.  For each fixed $\zeta\in S^{2n-1}$, let   $v(w): = u(w\zeta), w\in D_{r}$. Making use of Lemma \ref{hhl},  Lemma \ref{cgo} and Lemma \ref{fe}, we have  $v\in W^{1, 2}_{loc}(D_r)$,  vanishes to infinite order  in the $L^2$ sense at $0$ and   satisfies
$$|\bar\partial v(w) |= |\zeta \cdot \bar\partial u(w\zeta)|\le \frac{C}{|w|} |u(w\zeta)| =  \frac{C}{|w|} |v(w)|,\ \ w\in D_{r}. $$
 Thus, for a.e. $\zeta\in S^{2n-1}$  we can apply     Theorem \ref{mm1}   to get $v = 0$ on $D_{r}$.  Hence  $u=0$ on $B_{r}$. The weak unique continuation property    further applies to give $u\equiv 0$.    
\end{proof}

\bibliographystyle{alphaspecial}

\fontsize{11}{11}\selectfont

\vspace{0.7cm}
\noindent pan@pfw.edu,

\vspace{0.2 cm}

\noindent Department of Mathematical Sciences, Purdue University Fort Wayne, Fort Wayne, IN 46805-1499, USA.\\

\noindent zhangyu@pfw.edu,

\vspace{0.2 cm}

\noindent Department of Mathematical Sciences, Purdue University Fort Wayne, Fort Wayne, IN 46805-1499, USA.\\
\end{document}